\documentclass[10pt]{article}
\usepackage[margin=1.6in]{geometry}
\usepackage{amsmath,amssymb}
\usepackage{amsthm}
\usepackage{color, graphicx, tikz, tikz-cd}
\usepackage{subcaption}

\usepackage[backgroundcolor=yellow,textwidth=1.7in]{todonotes}

\newtheorem{thm}{Theorem}[section]
\newtheorem{lem}[thm]{Lemma}

\theoremstyle{definition}
\newtheorem{dfn}[thm]{Definition}
\theoremstyle{remark}

\newtheorem{rem}[thm]{Remark}
\newtheorem*{rem*}{Remark}

\newtheorem{example}[thm]{Example}

\newtheorem{THEO}{Theorem}

\usepackage[mathscr]{euscript} 

\newcommand{\mcl}[1]{\mathcal{#1}}

\newcommand{\cC}{\mcl{C}}

\newcommand{\ds}{\displaystyle}

\begin{document}

\date{}
\title{On algebras and matroids associated to undirected graphs}

\author{Boris Shapiro, Arkady Vaintrob}

\maketitle

\begin{abstract}
In this short note we make a few remarks on a class of  generalized incidence matrices whose matroids do not depend
on the orientation of the underlying graph and  natural commutative algebras associated to such matrices.
\end{abstract}
\section{Introduction}

\

\noindent
Two decades ago answering a geometrical question of Vl.~Arnold \cite{Ar} related to the Berry phases, the first author jointly with M.~Shapiro and A.~Postnikov introduced a certain algebra initially related  to a complete graph and later defined for any given undirected graph, see~\cite{SS, PSS, PS}. At the same time and independently, motivated by the necessities of the matroid theory D.~Wagner \cite{W1, W2} discovered the same algebra.


\medskip

The latter algebra is defined as follows. Let $G=(V,E)$ be a  undirected graph (multiple edges and loops are allowed) with the vertex set $V$ and the  edge set $E$.
For a field $K$ of zero characteristic, consider the square-free algebra defined as the quotient
\[B(E)=K[E]/(x_e^2), \ e\in E,\]
of the polynomial algebra in the \emph{edge variables}  $x_e, \ e\in E$, by the ideal generated by their squares.
\\[4pt]
Given an orientation $\sigma$ of $G$, we define the standard directed \emph{incidence matrix} $A(G,\sigma)$ of $G$ whose entries are given by
\begin{equation}\label{eq:incid}
  a_{v,e}=\begin{cases}
   -1, &  \text{if the edge $e$ begins at  $v$};\\
    1, & \text{if $e$ ends at  $v$};\\
    0, & \text{if $e$ is a loop or it is not incident to $v$.}
\end{cases}
\end{equation}
(The rows of $A(G,\sigma)$ are labeled by the vertices  and its columns by the edges of $G$).
Define the algebra
\begin{equation}
  \label{eq:alg}
  \cC(G,\sigma):=K[y_v]\subset B(E)
\end{equation}
generated by the elements
\begin{equation}
   \label{eq:gener}
  y_v =\sum_{e\in E} a_{v,e} x_e \in B(G), \mathrm{\ for\ }  v\in V.
\end{equation}

\begin{rem}
Observe that reversing the orientation of any edge $e\in E$ we simply change the sign of the corresponding generator $x_e$. So the isomorphism class of $\cC(G,\sigma)$ as a graded algebra does not depend on the choice of $\sigma$. Therefore we will denote this graded algebra by  $\cC(G)$ skipping $\sigma$.
\end{rem}

If we view each edge variable $x_e$ as representing the flow along the directed edge $e$  in $G$, then the variable $y_v$ can be interpreted as the total flow through the vertex $v$.
Thus $\cC(G)$ can be thought of as the algebra generated by the flows through the set of all vertices and, following~\cite{W1}, we call it the \emph{circulation algebra} of $G$.
\\[4pt]

One of the main results about $\cC(G)$ is the formula for  its Hilbert function
$$\ds h_{\cC(G)}(t):=\sum_{k\ge 0}\dim \cC^k(G)$$  obtained in~\cite{PS} and independently by D.~Wagner~\cite{W1}  which contains important information about the spanning subgraphs of $G$. (For  various graph- and matroid-related terminology and facts consult e.g. \cite{Ox}). In what follows, theorems, conjectures, etc., labeled by letters, 
are borrowed from the existing literature, while those labeled by numbers 
are hopefully new. 

\begin{dfn}\label{df1}Ê Let us fix a linear ordering of all edges of the graph $G$. 
For a spanning forest $F\subset G$, an edge $e \in G \setminus T$ is called \emph {externally active} for $F$ if there exists a cycle $C\subseteq G$ such that $e$ is the minimal edge of $C$ in the chosen ordering and $(C \setminus \{e\}) \subset T$. The \emph {external activity} of $F$ is the number of its externally active edges. 
\end{dfn}
Let $N_G^k$ denote the number of spanning forests $F \subset G$ of external activity $k$. Even though the notion of external activity depends on a particular choice of ordering of edges, the numbers $N_G^k$ are known to be independent of the latter choice.

\medskip
\begin{THEO}[\cite{PS, W1}]  \label{th:BASIC}The dimension $\dim_K \cC(G)$ of the algebra $\cC(G)$ 
is equal to  the number of spanning forests in $G$.
Additionally, the dimension $\dim_K \cC^k(G)$  of the $k$-th  graded component of $\cC(G)$  is
equal to the number  of spanning forests $S\subset E$ of $G$ with \emph{external activity}
$|E|-|S|-k$. Here $|\Omega|$ stands for the number of edges in a subgraph $\Omega$. 
\end{THEO}


Later G.~Nenashev~\cite{Ne} proved that the circulation algebras  $\cC(G_1)$ and $\cC(G_2)$ of two
graphs are isomorphic (as algebras) if and only if the usual graphical matroids of $G_1$ and $G_2$ are
isomorphic.



\medskip
To move further, notice that  D.~Wagner~\cite{W2} and Postnikov-Shapiro-Shapiro~\cite{PSS} introduced a more general
 class of algebras $\cC(A)$ defined similarly to  $\cC(G)$, but whose coefficients
$a_{v,e}$ in~(\ref{eq:gener})   are given by the entries of an arbitrary rectangular matrix $A=(a_{v,e})$ over the field
$K$.  We will call $\cC(A)$ the \emph{ circulation algebra} of the matrix $A$. In particular, observe that  $\cC(G)=\cC(A(G,\sigma))$, where $A(G,\sigma)$ is the directed incidence matrix of the graph $G$
corresponding to the orientation $\sigma$. The next definition is analogous to the Definition~\ref{df1}. 

\begin{dfn} Let $A \subset L$ be an arbitrary finite collection of vectors in a  linear space $L$ over the field $K$. Let us fix a linear ordering of  vectors in $A$. 
For an arbitrary subset  $I\subset A$ consisting of linearly independent vectors (independent subset), we call a vector $v \in A \setminus I$  \emph {externally active} for $I$ if there exists a dependent set of vectors $J\subseteq A$ such that $v$ is the minimal vector of $J$ in the chosen ordering and $J \setminus \{v\} \subset I$. The \emph {external activity} of  $I$  is, by definition,  the number of its externally active edges. 
\end{dfn}
Again let $N_A^k$ denote the number of  independent subsets $I \subset A$  having external activity $k$. Even though the notion of external activity depends on a particular choice of linear ordering of vectors, the numbers $N_I^k$ are known to be independent of this choice. The following analog of Theorem~\ref{th:BASIC} is valid.

\begin{THEO}[\cite{PSS}] \label{th:B}
The dimension $\dim_K \cC(A)$  is equal to the number of the independent subsets in (the matroid $M(A)$ of)  the vector configuration given by the columns of $A$.    Moreover, the dimension $\dim_K \cC^k(A)$ of the $k$-th graded component  $\cC^k(A)\subset  \cC(A)$ is equal to the number of independent subsets $I \subset  M(A)$ such that $k = |M(A)|- |I| -\mathrm{act}(I)$ where $|\omega|$ stands for the cardinality of a finite set $\omega$.

  \end{THEO}
  
  In~\cite{W2} D.~Wagner proved that the circulation algebras of two representable matroids are
isomorphic as \emph{graded algebras} if and only if the corresponding vector configurations are
projectively equivalent. 

The latter result of Wagner implies the former result of Nenashev using the unique representability of regular matroids
(i.e.~the fact that the vector configuration representing a regular matroid can be recovered from it up to
projective equivalence, see e.g.~\cite[Prop.~6.6.5]{Ox}) together with the simple fact that two finite-dimensional graded algebras generated
in degree one are isomorphic if and only if they are isomorphic as graded algebras. (In 
e.g.~\cite{BZ} the latter fact is proven in the more general set-up of finitely generated algebras).

\medskip
Other results about $\cC(G)$, $\cC(A)$ and some generalizations of these algebras can be found in  \cite{PS} and a number of follow-up papers on this topic. 
In this note we discuss  several analogs of $\cC(G)$ related to generalized incidence matrices as well as some properties of the corresponding matroids.

\medskip
\noindent
\emph{Acknowledgements.} The first author is sincerely grateful to the University of Oregon, Eugene for the hospitality in April 2019 when a part of this project was carried out.  The second author wants to acknowledge  the hospitality and financial support of Uppsala university in June-July 2019. The authors want to thank G.~Nenashev for his interest in this note and several useful comments.

\section{Results}

    As we mention above, the algebra $\cC(G)$ contains interesting graph-theoretical information about   $G$.
    Although its initial definition depends on a choice of  orientation $\sigma$ of the graph $G$, in the end $\sigma$ turns out to be
    irrelevant since different orientations give rise to isomorphic algebras. Below we will be interested in similar situations when either the  definition of an algebra related to $G$ does not use orientation of its edges or when all possible orientations lead to isomorphic algebras.

\subsection{Algebra  $\cC_+(G)$}
A natural algebra  similar to $\cC(G)$ and whose definition
  does not require any orientation  of edges of $G$ can be introduced as follows. 
    Denoting it by  $\cC_+(G)$ we define it exactly like  $\cC(G)$
    except for the formula for the coefficients   $a_{v,e}$ which now will  be  given  by
\[\ds  a_{v,e}=\begin{cases}
   1, &  \text{if the edge $e$ contains the vertex $v$};\\
   0, & \text{otherwise.}
\end{cases}
\]
 In other words,  $\cC_+(G)= \cC(A_+(G)), $
 where $A_+(G)$ is the standard \emph{undirected} incidence matrix of $G$.

 Similarly to $\cC(G)$, the graded algebra $\cC_+(G)$ contains interesting information about $G$.
 
 \begin{dfn}
(a) Given a graph $G=(V,E)$, we say that its (edge-induced) subgraph $F\subset E$ is an \emph{odd-circle pseudoforest}   if every connected component of $F$ is either a tree or a
unicycle whose cycle has odd length. (A \emph{unicycle}   is a graph containing a unique cycle, i.e. a connected graph obtained from a tree by adding exactly one edge connecting some of its vertices.)

\noindent
(b) A connected graph $G=(V,E)$ is called an \emph{even circuit} if it either is an even cycle or is a union of two odd cycles  sharing exactly one vertex or is a union of two disjoint odd cycles exactly two vertices of which are connected by a bridge. (The three types of even circuits are shown  in Figure~\ref{fig:evencirc}~a,b,c respectively; the second type can be thought as a degenerate case of the third type when the bridge contracts of a single vertex).

\end{dfn}

\begin{dfn}
Denote by $\mathcal{P}_+(G)$ the set of all odd-circle pseudoforests in $G$.  Given a linear ordering of the set $E$ of edges in $G$, we call an
edge $e$ \emph{evenly active} for a pseudoforest $F\in \mathcal{P}_+(G)$ if $F\cup \{e\}$ contains an
even circuit  in which $e$ is the smallest edge with respect to
the chosen ordering. 
\end{dfn}

  Our first result is as follows. %

\begin{thm}
  \label{thm:even-circ} Given an undirected graph $G=(V,E)$, one has the following. 

\smallskip
\noindent
{\rm 1.} The dimension $\dim_K \cC_+(G)$ of the algebra  $\cC_+(G)$ is equal to the number of spanning odd-circle
pseudoforests in $G$, i.e.
\[
  \dim_K \cC_+(G) = |\mathcal{P}_+(G)|.
\]

\smallskip
\noindent
{\rm 2.} 
 The dimension  $\dim_K \cC_+^k(G)$ of the $k$-th graded component of
 $\cC_+(G)$ is equal to the number of odd-circle pseudoforests $F\subseteq G$ whose
even activity equals $|E\setminus F|-k$, i.e. 
 \[
  \dim_K \cC_+^k(G)=\bigl| \{F\in \mathcal{P}_+(G) \,|\, \mathrm{act}_+(F)=|E \setminus  F| - k\}\bigr|~.
 \]
  \end{thm}

\begin{figure}
\centering
\tikzstyle{every node}=[circle, draw, fill=black!50, inner sep=0pt, minimum width=4pt]
%
\begin{subfigure}[b]{0.16\textwidth}  \centering \resizebox{\linewidth}{!}{
\begin{tikzpicture}[thick,scale=0.21]%
    \draw \foreach \x in {0,45,...,315} {  (\x:4) node{} -- (\x+45:4)    };
  \end{tikzpicture}
  }
\caption{even cycle}
\label{fig:a}
\end{subfigure}
\hfill
\begin{subfigure}[b]{0.24\textwidth}  \centering \resizebox{\linewidth}{!}{
\begin{tikzpicture}[thick,scale=0.21]%
   \draw \foreach \x in {0,72,...,288} {    (\x:4) node{} -- (\x+72:4)    };
   \draw (0:4) -- ++ (30:5) node{} -- ++ (270:5) node{} -- ++ (150:5) node{};
\end{tikzpicture}
}
\caption{odd figure-eight}
\label{fig:b}
\end{subfigure}
\hfill
\begin{subfigure}[b]{0.35\textwidth}  \centering \resizebox{\linewidth}{!}{
\begin{tikzpicture}[thick,scale=0.18]%
   \draw \foreach \x in {0, 51.4, 102.9, 154.3, 205.8, 257.2, 308.6, 360}
          {        (\x:4) node{} -- (\x+51.43:4)      };
   \draw (0:4) -- ++ (0:3) node{}  -- ++ (0:3) node{}   ;
\begin{scope}[shift={(0:15)}]
     \draw \foreach \x in {36,108,...,324} {      (\x:5) node{} -- (\x+72:5)     };
\end{scope}
\end{tikzpicture}
}
\caption{odd handcuff}
\label{fig:c}
\end{subfigure}
\caption{Even circuits}
\label{fig:evencirc}
\end{figure}
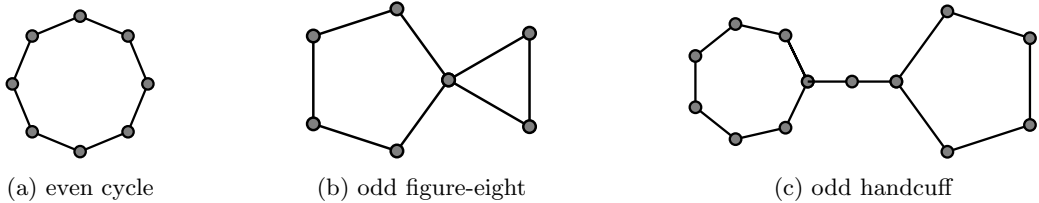



Theorem~\ref{thm:even-circ} is a direct consequence of the following lemma and Theorem~\ref{th:B}. 

\begin{lem}
A set of columns $S$ in the undirected incidence matrix $A_+(G)$  is dependent if and only
if the edge-induced subgraph of $G$ whose edges correspond to the columns in $S$ contains an even circuit.

\end{lem}

\begin{proof}  One can easily check that the set of columns corresponding to any odd-circle pseudoforest is linear independent. Indeed, an edge incident to a hanging vertex can not be a part of a linear dependence. Similarly, a single odd cycle gives no linear dependence of its edges/columns. 
On the other hand, all three types of even circuits provide linearly dependent sets of edges. These dependences are straightforward. For an even cycle, one has to assign the coefficients $\pm 1$ alternatingly. For the second type, see Figure~\ref{fig:evencirc}~b, one traverses this graph as an Eulerian assigning the coefficients $\pm 1$ alternatingly along the path. Finally, for the third type, see Figure~\ref{fig:evencirc}~c,  let us first choose one of two odd cycles. Then we traverse the even circuit starting at the vertex of the chosen cycle attached to the bridge  first going around the chosen cycle and assigning $\pm 1$ alternatingly on the way, then we follow  the bridge assigning the coefficients $\pm 2$ alternatingly, and, finally we traverse the second odd cycle  again assigning the coefficients $\pm 1$. By construction, the sum of the coefficients assigned to all  edges incident to any given vertex vanishes. 

To finish the proof, observe that adding an edge to an odd-circle  pseudoforest, we either obtain a new pseudoforest (in which case the set of edges/columns of the incidence matrix is still independent) or a graph containing an even circuit (in which case the set of edges is dependent). Indeed, if the new graph is not a pseudoforest, then it either contains an even cycle of two odd cycles connected by a bridge. The fact that 
removing any edge from a even circuit one obtains an   odd-circle  pseudoforest finishes the proof. 
\end{proof}

\begin{rem}
 The  matroid appearing in Theorem~\ref{thm:even-circ} is called the \emph{even-circle matroid} 
by M.~Doob~\cite{Do} and J.M.S.~Sim\~oes-Pereira~\cite{SP}.
  It  also coincides with a \emph{signed graphic matroid}  of  T.~Zaslavsky~\cite{Za1},
  which   corresponds to the case when all the edges of $G$ are given the negative sign.
 Additionally, the even-circle matroid of  $G$ is the same object as the \emph{factor matroid} in e.g.~\cite{W3}.
 It is  also worth  mentioning that the matrix $A_+(G)^T A_+(G)-2I$ equals
the adjacency matrix of the line graph of $G$ which explains the appearance of the even-circle
matroid in~\cite{Do}.
\end{rem}

\subsection{Algebra $\cC(G,\sigma,\gamma)$}
Another meaningful way to associate a graded algebra to an undirected graph $G$
 so that the dimensions of its graded components will contain non-trivial information about   $G$ is as follows.

To define it, we first transform $G$ into a \emph{gain graph}  by
choosing an orientation $\sigma$ of $G$ and a \emph{gain function} $\gamma:E\to  K\setminus \{0\}$, i.e. we assign a non-zero gain $\gamma(e)$ to each directed edge $e\in E$, see e.g.~\cite{Za2}.%

\begin{rem} Gain graphs appear in the study of flow networks with gains/losses, see~\cite[Ch.6]{GM}. They also play an important role in the topological graph theory which is  the study of 
embeddings of graphs and bundles on graphs, where they occur under the name \emph{voltage graphs}, see~\cite{Gr}.
They also form a special subclass of the so-called \emph{biased graphs} introduced in~\cite{Za2}.
\end{rem}

Now similarly to the directed incidence matrix $A(G,\sigma)$ given by~\eqref{eq:incid}, 
we introduce the \emph{incidence matrix}  $A_g(G,\sigma,\gamma)$ of
the gain graph $(G,\sigma,\gamma)$ which is  the $|V|\times |E|$-matrix
 with the entries: 
\begin{equation}
  \label{eq:gain}
  a_{v,e}=\begin{cases}
        0, &  \text{if the edge $e$ is not incident to $v$};\\
   -1, &  \text{if $e$ begins at  $v$};\\
    \gamma(e), & \text{if $e$ ends at  $v$};\\
    \gamma(e)-1, & \text{if $e$ is a loop}.
\end{cases}
\end{equation}
(In the theory of flow networks  $A_g(G,\sigma,\gamma)$ is interpreted as the matrix of additional factors such  that the flow along the edge $e$ is
multiplied by  $\gamma(e)$).

\begin{rem} 
Using the above notation, we see that for any choice of orientation $\sigma$ of $G$, one has 
$$A(G,\sigma)=A_g(G,\sigma,\gamma_+) \text{ and } A_+(G)=A_g(G,\sigma,\gamma_-)$$ where $\gamma_+$ (resp.\ $\gamma_-$)
is the constant gain function with value $\gamma_+(e)=1$ (resp.~$\gamma_+(e)=-1$) for every edge $e$. 
\end{rem}

The next definition is again similar to that of $\cC(G,\sigma)$.

\begin{dfn}
The \emph{circulation algebra $\cC(G,\sigma,\gamma)$ of a directed gain graph} $(G,\sigma,\gamma)$
is defined  by the formulas~\eqref{eq:alg}
and~\eqref{eq:gener}.
\end{dfn}

\begin{rem}
It is clear that if we change the direction of any edge $e\in E$ and, at the same time, replace
the value $\gamma(e)$ by its inverse $\gamma(e)^{-1}$, the obtained circulation algebra 
will be isomorphic to the initial one. This isomorphism is  induced by the variable change $x_e\mapsto -\gamma(e)^{-1} x_e$.
\end{rem}

\begin{dfn} Given an undirected graph $G$  with an orientation $\sigma$, a gain function $\gamma$, and a cycle $S\subset G$,    consider two oriented cycles $\overrightarrow S$ and $\overleftarrow S$ obtained by the choice of  one of two possible directions to traverse $S$. For  $\overrightarrow S$, define the  product 
 $\Gamma(\overrightarrow S)=\prod_{e\in S} \gamma^{\epsilon(e)}(e)$
where  $\epsilon(e)=1$  if the chosen direction to traverse $S$ gives the same orientation of $e$ as $\sigma$ and we take $\epsilon(e)=-1$ otherwise. (Similar expression defines $\Gamma(\overleftarrow S)$ and one has the obvious relation  $\Gamma(\overrightarrow S) \Gamma(\overleftarrow S)=1$.) The quantity $\Gamma(\overrightarrow S)$ (resp. $\Gamma(\overleftarrow S)$) will be called the \emph{gain} or the \emph{circulation} along the oriented cycle $\overrightarrow S$ (resp. $\overleftarrow S$). We say that  $S$ is \emph {gainless} if  $\Gamma(\overrightarrow S)=1$ (which is equivalent to $\Gamma(\overleftarrow S)=1$). 
\end{dfn}

\smallskip
\begin{example}
Observe that in general, the algebra $\cC(G,\sigma,\gamma)$ does depend on the choice of orientation $\sigma$ of $G$.
In particular, let $G$  be a $3$-cycle with gains $1, 2$, and $2$ of its edges.
If the edges with gain $2$ are directed in the same way, then $\dim_K \cC(G,\sigma,\gamma)=8$, while if these two edges are oppositely
directed, then $\dim_K \cC(G,\sigma,\gamma)=7$.
(Observe that in the latter case, $G$  is a gainless cycle.)
\end{example}
However, when the gains $\gamma(e)$ of all edges  are \emph{generic}, then, 
 the Hilbert series of the algebra $\cC(G,\sigma,\gamma)$ is independent of the orientation.
 More exactly, we have the following result.

\begin{thm}\label{th:bicircular}
Let $(G,\gamma)$ be a gain graph. If for every cycle $S\subset E$ and any subset $P\subset S$
of its edges, 
            \[ \prod_{e\in P}\gamma(e) \ne \ds \prod_{e\in S\setminus P}\gamma(e)~
            \]
then the Hilbert series of the circulation algebra $\cC(G,\sigma,\gamma)$ is independent of the orientation $\sigma$. In this case, the dimension $\dim_K \cC(G,\sigma,\gamma)$ is equal to the total number of  spanning
pseudoforests in $G$.
\end{thm}

\begin{proof} It is easy to check  %
that the independent sets of the matroid represented by the matrix
$A_g(G,\sigma,\gamma)$ are  exactly the pseudoforests without gainless cycles, see  e.g.~\cite{Za2}.   
Therefore, from~\cite{W2,PSS} it follows that the dimension of the algebra $\cC(G,\sigma,\gamma)$ is equal to the  number of spanning pseudoforests in $G$.
\end{proof}

\begin{rem}
Matroid appearing in Theorem~\ref{th:bicircular} coincides with the well-known \emph{bicircular matroid}, see e.g. \cite {Ma}.   Some results about graphs with isomorphic bicircular matroids can be found  in~\cite{CGW} (see also~\cite{Neu} and~\cite{SOSG}).
\end{rem} 

Theorem~\ref{th:bicircular} leads to further generalizations.
  For example, we may consider the so-called \emph{generalized incidence matrices}. 

\begin{dfn}
 A matrix  $A$ with entries
 in a field  $K$  is called a \emph{generalized incidence matrix} if its every column has at most two non-vanishing entries.
\end{dfn}

To each generalized incidence matrix $A$ we associate a  multigraph $G=G(A)$
whose vertices $v\in V(A)$ correspond the rows of $A$ and whose edges $E$ correspond to its non-zero columns. Namely, an edge corresponding to a column with two non-vanishing entries connects the vertices  corresponding to the respective rows while the edge corresponding to a column with a single non-vanishing entry is a loop at the vertex corresponding to the respective row.

If we choose an orientation $\sigma$ of $G(A)$, (e.g.  by fixing a linear ordering on $V$ and assigning to each edge the direction from the smaller incident vertex to the the larger incident vertex), 
then $G(A)$ becomes a gain graph as follows.
If $e=(a,b)\in E$ is an edge of $G(A)$ directed, say, from $v$ to $u$, then we
set $\gamma(e):=-A(u,e)/A(v,e)$. Additionally, if $e=(v)$ is a loop, then $\gamma(e):=A(v,e)+1$.

\smallskip

It is clear that the vector matroid represented by  $A$ is the same as the gain matroid of the directed gain graph 
$G(A)$.

\begin{rem}
All graph-related matroids which we discussed earlier are  represented by 
generalized incidence matrices and we can restate  the above theorems in these terms. 
\end{rem}

The problem of   describing generalized incidence matrices 
   which give rise to matroids which (up to an isomorphism) do not depend  on the orientation of their gain graphs seems to be quite natural.  
We will call such matroids \emph{orientation-independent}   and the final result of this note describes them.  


\begin{thm}\label{th:main}
 The matroid corresponding to a gain graph $(G,\gamma)$ with some orientation $\sigma$
   is orientation-independent if and only if
 for every cycle $S\subset E$ in $G$, we have that 
 \\[0.5pt]
 either
\\[0.5pt]
\noindent
  {\rm (1)}  for every $e\in S$, $\gamma(e)=\pm 1$;
\\[0.5pt]
      or,
 \\[0.5pt]
     \noindent
  {\rm (2)}     
       for any nonempty subset $P\subset S$, the product of the gains of its edges is distinct from that of the gains of its complement, i.e. 
            \[ \prod_{e\in P}\gamma(e) \ne \ds \prod_{e\in S\setminus P}\gamma(e)~.
            \]
    \end{thm}

\begin{proof}
Indeed, the independent sets of the matroid corresponding to $A_g(G,\sigma,\gamma)$ are given by the
pseudoforests without gainless cycles, see~\cite{Za2}. Change of the orientation of an edge $e$
corresponds to replacing $\gamma(e)$ by $(\gamma(e))^{-1}$. So changing the orientation of an edge
$e\in E$ we will obtain a different  gain matroid if and only if this edge belongs to a cycle $S$ which is gainless  
 for one orientation of $e$ and not gainless for the opposite
orientation. This implies that $\gamma(e)\ne \gamma(e)^{-1}$, i.e.\ $\gamma(e)\ne \pm 1$ and that
$e$ belongs to some cycle $S$ in $G$ which is gainless for at least one possible choice of orientations of the edges of $G$. The condition that $S$ never becomes gainless for whatever choice of orientations of edges is given by  (2) in the formulation of Theorem~\ref{th:main}.  
\end{proof}

\section{Outlook}
\label{sec:concl}

\noindent
{\bf 1.} The three types of ``cycle'' matroids (i.e. the standard cyclic, the bicircular, and the
  even-circle matroids) occur as the  special cases of the above orientation-independent matroids. 
  Is there  a way to decompose  an arbitrary orientation-independent matroid into matroids of these
 three types?  


\smallskip
\noindent
{\bf 2.} What can be said about two graphs $G_1$  and $G_2$ for which all three (or maybe only two)
  of the introduced algebras are isomorphic?%

 Notice that D.~Wagner~\cite{W3} characterizes all graphs $G$ for which $\cC_+(G)\simeq\cC_+(\Gamma)$, where
  $\Gamma$ is the fixed $4$-connected bipartite graph. (In this case, $\cC_+(\Gamma)\simeq\cC(G)$.)


\smallskip
\noindent
{\bf 3.} Is it possible to extend Theorem~\ref{th:main} from the case of generalized incidence matrices to arbitrary matrices $A$.

\smallskip
\noindent
{\bf 4.} Observe that gain graphs can be defined with with values of the gains belonging to an arbitrary group 
 (comp. with Dowling geometries), see.~\cite{Dw,Za2}. 
   If we replace $\pm 1$ by arbitrary group elements of order $2$,  our results should  more or less straightforwardly generalize to the case of abelian groups.  However a generalization to the case of non-abelian groups might be highly non-trivial.



\end{document}